\documentclass[ECP,preprint]{ejpecp}

\newcommand{\E}{\mathbb{E}}
\newcommand{\R}{\mathbb{R}}
\newcommand{\PR}{\mathbb{P}}
\newcommand{\uvec}{\mathbf{u}}
\newcommand{\pvec}{\mathbf{p}}
\newcommand{\qvec}{\mathbf{q}}
\newcommand{\hvec}{\mathbf{h}}
\newcommand{\PN}{\mathcal{P}_N}
\newcommand{\TT}{\mathcal{T}}
\DeclareMathOperator{\Cov}{Cov}

\SHORTTITLE{Equal probabilities maximize the expected sibling deficit}

\TITLE{Equal probabilities maximize the expected deficit\\
in the siblings of the coupon collector}

\AUTHORS{%
  Aristides~V.~Doumas\footnote{Department of Mathematics, School of Applied
    Mathematical and Physical Sciences, National Technical University of Athens,
    Zografou Campus, 15780 Athens, Greece; and Archimedes/Athena Research Center,
    Greece. E-mail:~\texttt{\href{mailto:adou@math.ntua.gr}{\nolinkurl{adou@math.ntua.gr}},\
\href{mailto:aris.doumas@hotmail.com}{\nolinkurl{aris.doumas@hotmail.com}}}}%
  \and
  S.~Spektor\footnote{School of Data, Computing and Mathematics, Canisius University,
    2001 Main Street, Buffalo, NY 14208-1098, USA.
    \EMAIL{spektors@canisius.edu}}}


\AMSSUBJ{60C05 ; 60E15 ; 05A20}

\ABSTRACT{%
In the siblings (or brotherhood) variant of the coupon collector's problem a main
collector draws coupons until her own album is complete, passing every duplicate
down a chain of siblings; the $j$th collector is then left with $U_j^N$ empty places,
$j\ge2$. Doumas and Papanicolaou conjectured that, for every fixed number of coupon
types $N$ and every $j\ge2$, the expected deficit $\E[U_j^N]$ is maximized by the
equiprobable distribution. We prove this in the stronger radial form that $\E[U_j^N]$
decreases strictly along every segment from the uniform vector toward an arbitrary
distribution with all $p_k>0$. Starting from the Poissonized integral for $\E[U_j^N]$, we
factor out the distribution function $\prod_i(1-e^{-p_it})$ of the Poissonized completion time and
differentiate along the radial segment. An integration by parts then turns the derivative
into a weighted covariance of the coordinate vector with the function $x\mapsto x^{\,j-1}$,
whose sign follows from Chebyshev's sum inequality. We also record a finite closed
form for $\E[U_j^N]$ as an alternating sum over subsets of coupon types.}

\begin{document}

\section{Introduction and main result}\label{sec:intro}

There are $N$ types of coupons, and each purchase yields type $k$ with probability
$p_k>0$, independently, with $\sum_{k=1}^N p_k=1$. A main collector buys until her
album is complete---until every type has appeared at least once---and we write $T_N$
for the number of purchases this requires. The duplicates she accumulates are not
thrown away but handed down a chain of siblings: the first repeat of a type goes to the
second collector, and in general the $m$th copy of each type goes to the $m$th collector.
Equivalently, at time $T_N$ the album of the $j$th collector holds type $k$ exactly
when type $k$ has been drawn at least $j$ times. We study the number of gaps still left
in that album,
\[
U_j^N=\#\bigl\{\,k\in\{1,\dots,N\}:\text{type $k$ drawn fewer than $j$ times by }T_N\,\bigr\},
\qquad j\ge2,
\]
with $U_1^N\equiv0$. The equiprobable case descends from the double Dixie cup problem of
Newman and Shepp \cite{NS}; the brotherhood model was studied by the symbolic calculus of
Foata, Han and Lass \cite{FHL} and Foata and Zeilberger \cite{FZ}. The
unequal-probability successive-album model, and the integral identity we start from, are
due to Adler, Oren and Ross \cite{AOR}; the quantity $U_j^N$ and the conjecture below
appear in Doumas and Papanicolaou \cite{DPsiblings}.

Throughout, $\PN=\{\pvec\in\R^N:p_k>0,\ \sum_{k=1}^N p_k=1\}$ is the open probability
simplex and $\uvec=(1/N,\dots,1/N)$ the uniform vector. We use the standard continuous-time
embedding: let each type $i$ arrive according to an independent Poisson process of rate
$p_i$. The embedded jump chain---the sequence of types in order of arrival---has exactly the
discrete purchase law, and the time $\widetilde T=\max_i\tau_i$ of the $T_N$th jump, where
$\tau_i$ is the first arrival of type $i$, is the continuous-time completion instant.
In this embedding, type $k$ is counted in $U_j^N$ exactly when its $j$th arrival, at a time
with the $\mathrm{Gamma}(j,p_k)$ law (shape $j$, rate $p_k$), occurs after every $\tau_i$ with $i\ne k$ (for $j\ge2$
this already forces it to occur after $\tau_k$). Writing the density of the $j$th arrival of
$k$ and using the independence of the processes gives (Adler, Oren and Ross \cite{AOR}; see
also Brayton \cite{Brayton})
\begin{equation}\label{eq:DP}
\E[U_j^N]=\sum_{k=1}^N\int_0^\infty p_k e^{-p_k t}\,
\frac{(p_k t)^{j-1}}{(j-1)!}\,\prod_{i\ne k}\bigl(1-e^{-p_i t}\bigr)\,dt,
\qquad j\ge2 .
\end{equation}

Equiprobability is a recurring extremum in this family of problems: among all coupon
distributions it minimizes the expected completion time---a Schur-convexity fact that
follows from the majorization theory of Marshall, Olkin and Arnold \cite{MOA}---and it
minimizes the variance of the completion time, conjectured by Doumas and Papanicolaou
\cite{DP2012} and proved for one complete set by Yu \cite{Yu} and for several sets by
Long \cite{LongVar}. Asymptotics of these quantities for general $\pvec$ are studied in
\cite{DP2012,DP2016}. The siblings model carries the analogous conjecture.

\begin{conjecture}[Doumas and Papanicolaou \cite{DPsiblings}]\label{conj:main}
For every $N\ge2$ and every integer $j\ge2$, the map $\pvec\mapsto\E[U_j^N]$ on $\PN$
attains its maximum at the uniform vector $\uvec$, and only there.
\end{conjecture}

The standard majorization tools do not apply here: as we note in Remark~\ref{rem:schur},
$\E[U_j^N]$ is already not Schur-concave at $N=3$, $j=2$, so the Schur--Ostrowski criterion
is unavailable.

\begin{theorem}\label{thm:main}
Fix $N\ge2$ and an integer $j\ge2$. For $\pvec\in\PN$ with $\pvec\ne\uvec$, let
$\qvec(\theta)=\uvec+\theta(\pvec-\uvec)$, $\theta\in[0,1]$, be the segment from $\uvec$
to $\pvec$. Then
\[
\frac{d}{d\theta}\,\E[U_j^N]\bigl(\qvec(\theta)\bigr)<0
\qquad\text{for every }\theta\in(0,1].
\]
Hence $\E[U_j^N]$ strictly decreases along the segment, so
$\E[U_j^N](\pvec)<\E[U_j^N](\uvec)$ for all $\pvec\in\PN\setminus\{\uvec\}$, and $\uvec$
is the unique maximizer of $\E[U_j^N]$ on $\PN$.
\end{theorem}

Combined with the closed form of Section~\ref{sec:id}, Theorem~\ref{thm:main} yields the
explicit sharp bounds $1<\E[U_j^N]\le\sum_{m=1}^N(-1)^{m-1}\binom{N}{m}m^{\,1-j}$
(Corollary~\ref{cor:range}).

Section~\ref{sec:id} gives two exact rewritings of \eqref{eq:DP}: a finite alternating
sum over subsets (Proposition~\ref{prop:closed}) and a one-dimensional integral
(Proposition~\ref{prop:integral}) that isolates the completion-time factor
$\prod_i(1-e^{-p_it})$. Section~\ref{sec:proof} proves Theorem~\ref{thm:main} by
differentiating the latter; Section~\ref{sec:hessian} computes the Hessian at $\uvec$; and
Section~\ref{sec:remarks} collects remarks on the failure of Schur-concavity, the relation
to other work, and the extension to real $j$. A preprint version of this work is
\cite{DSself}; Conjecture~\ref{conj:main} was subsequently also proved, by a different
method, in the preprints of Long \cite{LongLimit,LongRadial}, which cite \cite{DSself}. We
compare the approaches in Remark~\ref{rem:long}.

\section{Two exact identities}\label{sec:id}

For $x>0$ and $t>0$ set
\begin{equation}\label{eq:defs}
r(x):=\frac{1}{e^{x}-1},\qquad
g_j(x):=\frac{x^{\,j}e^{-x}}{(j-1)!\,(1-e^{-x})}=\frac{x^{\,j}}{(j-1)!}\,r(x),
\qquad
\Phi(t):=\prod_{i=1}^N\bigl(1-e^{-p_i t}\bigr).
\end{equation}
The function $r$ is positive and strictly decreasing on $(0,\infty)$, and the second
form of $g_j$ records that $r(x)=e^{-x}/(1-e^{-x})$. Under the Poissonization, $\Phi(t)$ is
the distribution function of the completion time $\widetilde T=\max_i\tau_i$, which is why
it is the natural factor to separate out.

\begin{proposition}\label{prop:closed}
For every $\pvec\in\PN$ and integer $j\ge2$,
\begin{equation}\label{eq:closed}
\E[U_j^N]
=\sum_{\varnothing\ne A\subseteq\{1,\dots,N\}}(-1)^{|A|-1}\,
\frac{\sum_{i\in A}p_i^{\,j}}{\bigl(\sum_{i\in A}p_i\bigr)^{j}} .
\end{equation}
\end{proposition}

\begin{proof}
By inclusion--exclusion,
\[
\prod_{i\ne k}\bigl(1-e^{-p_it}\bigr)
=\sum_{S\subseteq\{1,\dots,N\}\setminus\{k\}}(-1)^{|S|}e^{-p_St},
\qquad p_S:=\sum_{i\in S}p_i,
\]
and integrating term by term with
$\int_0^\infty t^{\,j-1}e^{-(p_k+p_S)t}\,dt=(j-1)!\,(p_k+p_S)^{-j}$ (legitimate since
$p_k+p_S\ge p_k>0$) gives
\[
\E[U_j^N]=\sum_{k=1}^N\sum_{S\subseteq\{1,\dots,N\}\setminus\{k\}}
(-1)^{|S|}\frac{p_k^{\,j}}{(p_k+p_S)^{\,j}} .
\]
Write $A=\{k\}\cup S$, so that $|S|=|A|-1$ and $p_k+p_S=p_A:=\sum_{i\in A}p_i$. Summing
over the choice of the distinguished element $k\in A$ yields \eqref{eq:closed}.
\end{proof}

The subset form \eqref{eq:closed} also appears in Long
\cite{LongLimit}. Evaluating it at $\uvec$, where $p_A=|A|/N$ and $M_A:=\sum_{i\in A}p_i^{\,j}=|A|/N^{\,j}$,
each term $M_A/p_A^{\,j}$ equals $|A|^{\,1-j}$, so Theorem~\ref{thm:main} becomes an explicit
sharp inequality.

\begin{corollary}\label{cor:range}
For every $N\ge2$, every integer $j\ge2$ and every $\pvec\in\PN$,
\[
1<\E[U_j^N](\pvec)\le\sum_{m=1}^{N}(-1)^{m-1}\binom{N}{m}\,m^{\,1-j},
\]
the upper bound attained only at $\uvec$.
\end{corollary}

\begin{proof}
The upper bound and its attainment are Theorem~\ref{thm:main} together with the value of
\eqref{eq:closed} at $\uvec$ computed above. For the lower bound, $U_j^N\ge1$ always: the
last coupon type to appear by $T_N$ has at that instant been drawn exactly once, hence is
deficient for $j\ge2$. Moreover $\PR(U_j^N\ge2)>0$---on the event that the last two new types
each arrive once by $T_N$---so $\E[U_j^N]>1$ strictly.
\end{proof}

\begin{remark}[sharpness of the lower bound]\label{rem:lower}
The bound $\E[U_j^N]>1$ cannot be improved: along $\pvec_\varepsilon=(1-\varepsilon-\cdots-\varepsilon^{N-1},\varepsilon,\dots,\varepsilon^{N-1})$,
whose small coordinates vanish at \emph{separated} scales, \eqref{eq:closed} gives
$\E[U_j^N](\pvec_\varepsilon)\to1$ as $\varepsilon\to0$ (each proper subset sum
$p_A^{\,-j}\sum_{i\in A}p_i^{\,j}$ either tends to $1$, when $A$ contains the dominant
coordinate, or is $o(1)$). Sending the small coordinates to $0$ at the \emph{same} scale
instead, as in $(1-(N-1)\varepsilon,\varepsilon,\dots,\varepsilon)$, collapses the problem
onto a uniform $(N-1)$-type face and the value tends to the maximum there, not to $1$; this
is why the scales must be separated. For $N=2$ the full range is $\bigl(1,\,2-2^{\,1-j}\bigr]$,
since $\E[U_j^2](p,1-p)=2-(p^{\,j}+(1-p)^{\,j})\to1$ as $p\to0$ or $1$. Finally, for $j=2$ the
upper bound is the harmonic number: $\sum_{m=1}^N(-1)^{m-1}\binom{N}{m}/m=H_N$ (a standard
identity), so $1<\E[U_2^N]\le H_N$.
\end{remark}

The next identity is the one we differentiate.

\begin{proposition}\label{prop:integral}
For every $\pvec\in\PN$ and integer $j\ge2$,
\begin{equation}\label{eq:integral}
\E[U_j^N]=\int_0^\infty \frac{\Phi(t)}{t}\sum_{k=1}^N g_j(p_k t)\,dt .
\end{equation}
\end{proposition}

\begin{proof}
In the $k$th summand of \eqref{eq:DP},
\[
p_k e^{-p_k t}\frac{(p_k t)^{j-1}}{(j-1)!}=\frac1t\,\frac{(p_k t)^{j}e^{-p_k t}}{(j-1)!},
\qquad
\prod_{i\ne k}\bigl(1-e^{-p_i t}\bigr)=\frac{\Phi(t)}{1-e^{-p_k t}},
\]
so the summand is $\tfrac1t\,\Phi(t)\,g_j(p_k t)$ by \eqref{eq:defs}. Summing over $k$
gives \eqref{eq:integral}. The integrand is integrable: near $t=0$, $\Phi(t)=O(t^N)$ and
$g_j(p_kt)/t=O(t^{\,j-2})$, and for large $t$ each $g_j(p_kt)$ decays exponentially.
\end{proof}

In \eqref{eq:integral} the symmetric factor $\Phi(t)=\prod_i(1-e^{-p_it})$ multiplies the
sum $\sum_k g_j(p_kt)$, in which each coordinate contributes a single term. The proof
below differentiates this form.

\begin{remark}[the case $j=1$]\label{rem:j1}
Formulas \eqref{eq:closed} and \eqref{eq:integral} are algebraic identities valid for
$j=1$ as well, but for $j=1$ their common value is the constant $1$, not
$\E[U_1^N]\equiv0$: the derivation of \eqref{eq:DP} uses $j\ge2$ (the $j$th arrival of $k$
then automatically follows $\tau_k$). We return to the degenerate case $j=1$ after
Lemma~\ref{lem:radial}.
\end{remark}

\section{Proof of the main theorem}\label{sec:proof}

Fix $\pvec\in\PN$ with $\pvec\ne\uvec$, and set $\hvec=\pvec-\uvec$, so that
$\sum_a h_a=0$ and $\hvec\ne0$. For $\theta\in[0,1]$ let $\qvec(\theta)=\uvec+\theta\hvec$,
with coordinates
\[
q_a=q_a(\theta)=\frac1N+\theta h_a=\frac{1-\theta}{N}+\theta p_a .
\]
Each $q_a(\theta)=\tfrac{1-\theta}{N}+\theta p_a$ is a convex combination of $\tfrac1N$ and
$p_a$, so it lies between them. Setting
\[
\delta:=\min_a p_a>0 ,
\]
and using $\min_a p_a\le\tfrac1N$, we have $q_a(\theta)\ge\min(\tfrac1N,p_a)\ge\delta$ for
every $a$ and every $\theta\in[0,1]$; also $q_a(\theta)\le1$. Thus $\qvec(\theta)\in\PN$ for
all $\theta\in[0,1]$. Write $\psi(\theta):=\E[U_j^N]\bigl(\qvec(\theta)\bigr)$. Below, the scalar
$W(t)=\sum_a r(q_at)$ and $G(t)=\sum_k g_j(q_kt)$ introduced in Lemma~\ref{lem:radial} and
\eqref{eq:psiprime}, and the product $\Phi$, are all understood to be evaluated at
$\qvec(\theta)$ (the one-variable
functions $r$ and $g_j$ of \eqref{eq:defs} do not depend on the vector).

We first record the bound that justifies differentiating \eqref{eq:integral} under the
integral sign, with $\delta=\min_a p_a$ as above.

\begin{lemma}[domination]\label{lem:dom}
Let $\mathcal F(t,\theta)=\tfrac1t\,\Phi(t)\sum_k g_j(q_kt)$ be the integrand of \eqref{eq:integral}.
Then, for all $t>0$ and all $\theta\in[0,1]$,
\[
|\partial_\theta \mathcal F(t,\theta)|\le C(1+t)\,\mathcal F(t,\theta),\qquad
\mathcal F(t,\theta)\le
\begin{cases}
C'\,t^{\,N+j-2}, & 0<t\le1,\\[2pt]
C'\,t^{\,j-1}e^{-\delta t}, & t\ge1,
\end{cases}
\]
for constants $C,C'$ depending on $N,j$ and on $\delta=\min_a p_a$; hence $\partial_\theta \mathcal F$
is dominated by a function integrable on $(0,\infty)$, uniformly in $\theta\in[0,1]$.
\end{lemma}

\begin{proof}
Write $\mathcal F_k=\tfrac1t\,\Phi(t)g_j(q_kt)$ for the $k$th summand, so $0\le \mathcal F_k\le \mathcal F$. By the
product rule, $\partial_\theta \mathcal F$ is a finite sum of terms, each some $\mathcal F_k$ times a factor
$t\,r(q_it)\,h_i$ or $t\,g_j'(q_kt)h_k/g_j(q_kt)$, and $|h_a|\le1$. Two elementary bounds
control these factors. From $e^{x}-1\ge x$ we get $r(x)\le1/x$, hence
$t\,r(q_it)\le1/q_i\le1/\delta$. And since $g_j'(x)/g_j(x)=\tfrac{j}{x}-\tfrac1{1-e^{-x}}$,
\[
t\,\frac{g_j'(q_kt)}{g_j(q_kt)}=\frac{j}{q_k}-\frac{t}{1-e^{-q_kt}},
\]
which tends to $(j-1)/q_k$ as $t\to0$. The elementary inequality $x/(1-e^{-x})\le1+x$
(for $x>0$) gives $t/(1-e^{-q_kt})=\tfrac1{q_k}\cdot\tfrac{q_kt}{1-e^{-q_kt}}\le\tfrac1{q_k}+t\le\tfrac1\delta+t$,
so $|t\,g_j'(q_kt)/g_j(q_kt)|\le \tfrac{j}{\delta}+\tfrac1\delta+t$.
Both factors are therefore bounded by $C(1+t)$ for all $\theta\in[0,1]$, giving
$|\partial_\theta \mathcal F|\le C(1+t)\mathcal F$. Finally, since $q_a(\theta)\le1$, we have
$\Phi(t)\le\prod_i q_it=O(t^N)$ and $\tfrac1t\sum_kg_j(q_kt)=O(t^{\,j-2})$, giving
$\mathcal F\le C'\,t^{\,N+j-2}$ for $0<t\le1$, while for $t\ge1$ each
$g_j(q_kt)\le C''\,t^{\,j}e^{-\delta t}$ gives $\mathcal F\le C'\,t^{\,j-1}e^{-\delta t}$;
the constants $C,C'$ depend only on $N,j$ and $\delta$, and all bounds are uniform in
$\theta\in[0,1]$.
\end{proof}

\begin{lemma}\label{lem:radial}
Let $j\ge1$. For every $\theta\in(0,1]$,
\begin{equation}\label{eq:final}
\psi'(\theta)=-\frac{1}{\theta N}\int_0^\infty \frac{t^{\,j}}{(j-1)!}\,
\Phi(t)\,W(t)^2\,\Cov_{w_t}\!\bigl(q,\;q^{\,j-1}\bigr)\,dt,
\end{equation}
where, for each $t>0$, the weights $w_{t,a}=r(q_a t)>0$ have total mass
$W(t)=\sum_{a=1}^N w_{t,a}$, and $\Cov_{w_t}$ is the covariance under the probability
vector $\bigl(w_{t,a}/W(t)\bigr)_{a=1}^N$; here $q$ and $q^{\,j-1}$ denote the coordinate
functions $a\mapsto q_a$ and $a\mapsto q_a^{\,j-1}$, so that
\[
\Cov_{w_t}(q,q^{j-1})
=\sum_a\frac{w_{t,a}}{W(t)}\,q_a^{\,j}
-\Bigl(\sum_a\frac{w_{t,a}}{W(t)}\,q_a\Bigr)
 \Bigl(\sum_a\frac{w_{t,a}}{W(t)}\,q_a^{\,j-1}\Bigr).
\]
\end{lemma}

\begin{proof}
\emph{Step 1 (differentiate).}
By Lemma~\ref{lem:dom} the integrand $\mathcal F$ of \eqref{eq:integral} may be differentiated under
the integral sign in $\theta$, for $\theta\in[0,1]$. Using $\partial_\theta\log\Phi(t)=t\sum_i h_i\,r(q_i t)$ and
$\partial_\theta g_j(q_a t)=t\,h_a\,g_j'(q_a t)$,
\begin{equation}\label{eq:psiprime}
\psi'(\theta)=\int_0^\infty \Phi(t)\Bigl[\,G(t)\sum_a h_a\,r(q_a t)
+\sum_a h_a\,g_j'(q_a t)\Bigr]dt,
\qquad G(t):=\sum_{k}g_j(q_k t).
\end{equation}

\emph{Step 2 (integrate by parts).}
The integrand of \eqref{eq:psiprime} has no definite sign in general (see
Remark~\ref{rem:signchange}), so to expose the sign we
remove $g_j'$ from the second sum: write $g_j'(q_a t)=\tfrac1{q_a}\,\partial_t g_j(q_a t)$
and integrate by parts in $t$:
\[
\int_0^\infty \Phi(t)\sum_a\frac{h_a}{q_a}\,\partial_t g_j(q_a t)\,dt
=\Bigl[\Phi(t)\sum_a\frac{h_a}{q_a}\,g_j(q_a t)\Bigr]_0^\infty
-\int_0^\infty \Phi'(t)\sum_a\frac{h_a}{q_a}\,g_j(q_a t)\,dt .
\]
The boundary term vanishes at both ends: as $t\to\infty$, $\Phi(t)\to1$ while each
$g_j(q_a t)\sim(q_a t)^{j}e^{-q_a t}/(j-1)!\to0$; as $t\to0^+$,
$\Phi(t)\sim t^{N}\prod_i q_i$ and $g_j(q_a t)\sim(q_a t)^{\,j-1}/(j-1)!$, so the product
is $O(t^{\,N+j-1})\to0$. With
$\Phi'(t)=\Phi(t)\sum_i q_i\,r(q_i t)$, the second sum in \eqref{eq:psiprime} contributes
\[
-\int_0^\infty \Phi(t)\Bigl(\sum_i q_i\,r(q_i t)\Bigr)
\Bigl(\sum_a\frac{h_a}{q_a}\,g_j(q_a t)\Bigr)dt ,
\]
and therefore $\psi'(\theta)=\int_0^\infty\Phi(t)\,K(t)\,dt$ with
\begin{equation}\label{eq:K}
K(t)=G(t)\sum_a h_a\,r(q_a t)
-\Bigl(\sum_i q_i\,r(q_i t)\Bigr)\Bigl(\sum_a\frac{h_a}{q_a}\,g_j(q_a t)\Bigr).
\end{equation}

\emph{Step 3 (identify the covariance).}
Substitute $g_j(x)=x^{\,j}r(x)/(j-1)!$. Fix $t>0$ and write $w_a=w_{t,a}=r(q_a t)$ and
$c=c(t)=t^{\,j}/(j-1)!$; then $g_j(q_a t)=c\,q_a^{\,j}w_a$, so $G(t)=c\sum_k q_k^{\,j}w_k$ and
$\tfrac{h_a}{q_a}g_j(q_a t)=c\,h_a q_a^{\,j-1}w_a$, whence
\[
K(t)=c\Bigl[\Bigl(\sum_k q_k^{\,j}w_k\Bigr)\Bigl(\sum_a h_a w_a\Bigr)
-\Bigl(\sum_i q_i w_i\Bigr)\Bigl(\sum_a h_a q_a^{\,j-1}w_a\Bigr)\Bigr].
\]
For a function $f$ write $\E_w[f]:=\tfrac1W\sum_a w_a f(q_a)$, with $W=\sum_a w_a$, and
abbreviate $\E_w[q^{\,s}]:=\tfrac1W\sum_a w_a q_a^{\,s}$. From $q_a=\tfrac1N+\theta h_a$ we
have $h_a=(q_a-\tfrac1N)/\theta$, so
\[
\sum_a h_a w_a=\frac{W}{\theta}\Bigl(\E_w[q]-\tfrac1N\Bigr),\qquad
\sum_a h_a q_a^{\,j-1}w_a=\frac{W}{\theta}\Bigl(\E_w[q^{\,j}]-\tfrac1N\E_w[q^{\,j-1}]\Bigr),
\]
and, since $G(t)=cW\,\E_w[q^{\,j}]$ and $\sum_i q_i w_i=W\,\E_w[q]$,
\[
\frac{K(t)}{c\,W^2/\theta}
=\E_w[q^{\,j}]\Bigl(\E_w[q]-\tfrac1N\Bigr)-\E_w[q]\Bigl(\E_w[q^{\,j}]-\tfrac1N\E_w[q^{\,j-1}]\Bigr)
=\frac1N\Bigl(\E_w[q]\,\E_w[q^{\,j-1}]-\E_w[q^{\,j}]\Bigr),
\]
the terms $\E_w[q^{\,j}]\,\E_w[q]$ cancelling. The right-hand side is
$-\tfrac1N\Cov_w(q,q^{\,j-1})$. Multiplying by
$\dfrac{c\,W(t)^2}{\theta}=\dfrac{t^{\,j}\,W(t)^2}{\theta\,(j-1)!}$
and integrating against $\Phi$ gives \eqref{eq:final}.
\end{proof}

\begin{remark}[the degenerate case $j=1$]\label{rem:j1deg}
Lemma~\ref{lem:radial} is stated for $j\ge1$ because its identity \eqref{eq:final} is
derived entirely from \eqref{eq:integral}, with no appeal to $j\ge2$. At $j=1$ the vector
$q^{\,0}$ is constant, so $\Cov_{w_t}(q,q^{0})=0$ and hence $\psi'\equiv0$: consistent with
Remark~\ref{rem:j1}, where \eqref{eq:closed} and \eqref{eq:integral} were seen to be the
constant $1$ at $j=1$.
\end{remark}

\begin{remark}[the integrand of \eqref{eq:psiprime} is sign-indefinite]\label{rem:signchange}
The integration by parts is what makes the sign visible: the bracket in \eqref{eq:psiprime}
need not have a fixed sign in $t$. For instance, at $\pvec=(0.6,0.25,0.15)$, $j=2$,
$\theta=0.6$, the integrand of \eqref{eq:psiprime} is negative for $t<t_0$ and positive for
$t>t_0$, with $t_0\approx11.0$; only after the integration by parts does the integrand
\eqref{eq:final} become a single covariance of one sign.
\end{remark}

\begin{proof}[Proof of Theorem~\ref{thm:main}]
Fix $j\ge2$. The weights $w_{t,a}=r(q_a t)$ are strictly positive, since $r>0$. Writing
$W=W(t)$ and $w_a=w_{t,a}$, expanding the product gives
\[
\sum_{a,b}w_aw_b\,(q_a-q_b)\bigl(q_a^{\,j-1}-q_b^{\,j-1}\bigr)
=2W\sum_a w_a q_a^{\,j}-2\Bigl(\sum_a w_a q_a\Bigr)\Bigl(\sum_a w_a q_a^{\,j-1}\Bigr),
\]
so, dividing by $2W^2$,
\[
\Cov_{w_t}(q,q^{\,j-1})
=\frac{1}{2W^2}\sum_{a,b}w_aw_b\,(q_a-q_b)\bigl(q_a^{\,j-1}-q_b^{\,j-1}\bigr).
\]
This is the weighted form of Chebyshev's sum inequality: because $x\mapsto x^{\,j-1}$ is strictly
increasing on $(0,\infty)$, each summand $(q_a-q_b)(q_a^{\,j-1}-q_b^{\,j-1})\ge0$, strictly
positive whenever $q_a\ne q_b$, so $\Cov_{w_t}(q,q^{\,j-1})\ge0$, with equality only when
all $q_a$ coincide. For $\pvec\ne\uvec$ and $\theta\in(0,1]$ the coordinates
$q_a(\theta)=\tfrac1N+\theta h_a$ are not all equal, so the covariance is strictly positive
for every $t>0$. In \eqref{eq:final} the factors $\tfrac{1}{\theta N}$,
$\tfrac{t^{\,j}}{(j-1)!}$, $\Phi(t)$ and $W(t)^2$ are all positive, so the integrand
$\tfrac{t^{\,j}}{(j-1)!}\Phi(t)W(t)^2\Cov_{w_t}(q,q^{\,j-1})$ is strictly positive for every
$t>0$; with the leading minus sign this gives $\psi'(\theta)<0$ for every $\theta\in(0,1]$.

Since $\psi$ is a rational function of $\theta$ (by \eqref{eq:closed}), hence continuous on $[0,1]$, and
$\psi'<0$ on $(0,1]$, integrating gives $\psi(1)-\psi(0)=\int_0^1\psi'(\theta)\,d\theta<0$;
thus $\E[U_j^N](\pvec)=\psi(1)<\psi(0)=\E[U_j^N](\uvec)$. As $\pvec\in\PN\setminus\{\uvec\}$
was arbitrary, $\uvec$ is the unique maximizer.
\end{proof}

\section{The Hessian at the uniform vector}\label{sec:hessian}

Theorem~\ref{thm:main} places the maximum at $\uvec$, but it does not by itself show
that the maximum is nondegenerate; a function can decrease along every ray from a point
and still have a vanishing Hessian there (for instance $x\mapsto-\|x\|^4$). We therefore
compute the Hessian directly. We work with \eqref{eq:closed} extended to the open positive
orthant $\R^N_{>0}$; this extension is homogeneous of degree $0$ (each term
$M_A/p_A^{\,j}$ is), so its restriction to $\PN$ is unambiguous. First, $\uvec$ is a critical
point of this restriction: by permutation symmetry, $\nabla\E[U_j^N]$ at $\uvec$ is invariant
under all coordinate permutations, hence proportional to $(1,\dots,1)$, so it vanishes on the
tangent space $\TT$. By the same symmetry the Hessian at $\uvec$ restricted to $\TT$ is
determined by a single scalar.

\begin{proposition}\label{cor:hess}
Let $\TT=\{\hvec\in\R^N:\sum_i h_i=0\}$ be the tangent space of $\PN$ at $\uvec$. The
Hessian of $\E[U_j^N]$ at $\uvec$, restricted to $\TT$, equals $\lambda_{N,j}\,I_\TT$,
where $I_\TT$ is the identity on $\TT$ and
\begin{equation}\label{eq:lambda}
\lambda_{N,j}=-\,N^3\,j(j-1)\sum_{m=2}^{N}(-1)^{m}\binom{N-2}{m-2}\frac{1}{m^{\,j+1}} .
\end{equation}
In particular $\lambda_{N,j}<0$, so $\uvec$ is a nondegenerate strict local maximizer of
$\E[U_j^N]$ on $\PN$.
\end{proposition}

For $j=2$ the prefactor is $N^3\cdot2$, and \eqref{eq:lambda} evaluates to
$\lambda_{3,2}=-54\bigl(\tfrac18-\tfrac1{27}\bigr)=-\tfrac{19}{4}$ and
$\lambda_{4,2}=-128\bigl(\tfrac18-\tfrac2{27}+\tfrac1{64}\bigr)=-\tfrac{230}{27}$.

\begin{proof}
Let $H$ be the Hessian at $\uvec$ of the function \eqref{eq:closed}, regarded as a smooth
function on the open positive orthant $\R^N_{>0}$ in the coordinates $p_1,\dots,p_N$. Its
restriction to $\TT$ is independent of this choice of extension off the simplex, because the
simplex is affine with tangent space $\TT$. By permutation symmetry its diagonal entries
share a common value
$H_{11}$ and its off-diagonal entries a common value $H_{12}$, so
$H=(H_{11}-H_{12})\,I+H_{12}\,J$, where $J$ is the all-ones matrix. For $\hvec\in\TT$ one
has $J\hvec=0$, hence $H\hvec=(H_{11}-H_{12})\hvec$; that is, $H$ restricted to $\TT$ is
$\lambda_{N,j}I_\TT$ with $\lambda_{N,j}=H_{11}-H_{12}$.

Write $F_A(\pvec)=M_A/p_A^{\,j}$ for a nonempty $A$, with $p_A=\sum_{i\in A}p_i$ and
$M_A=\sum_{i\in A}p_i^{\,j}$, so that $\E[U_j^N]=\sum_A(-1)^{|A|-1}F_A$. A direct
computation, evaluated at $p_i=1/N$ (so $p_A=|A|/N$, $M_A=|A|/N^{\,j}$), gives
\[
\partial^2_{p_a p_a}F_A=\frac{N^2 j\,(j-1)(|A|-1)}{|A|^{\,j+1}}\quad(a\in A),\qquad
\partial^2_{p_a p_b}F_A=-\frac{N^2 j\,(j-1)}{|A|^{\,j+1}}\quad(a\ne b,\ a,b\in A),
\]
and $\partial^2_{p_a p_b}F_A=0$ if $a\notin A$ or $b\notin A$. For a fixed pair $(a,b)$, the
number of $A$ with $|A|=m$ containing $a$ is $\binom{N-1}{m-1}$, and the number
containing both $a$ and $b$ is $\binom{N-2}{m-2}$. Hence
\[
H_{11}-H_{12}
=N^2 j(j-1)\sum_{m=1}^N\frac{(-1)^{m-1}}{m^{\,j+1}}
\Bigl[(m-1)\binom{N-1}{m-1}+\binom{N-2}{m-2}\Bigr].
\]
Since $(m-1)\binom{N-1}{m-1}=(N-1)\binom{N-2}{m-2}$, the bracket is $N\binom{N-2}{m-2}$,
which vanishes for $m=1$. Therefore
\[
\lambda_{N,j}=-\,N^3 j(j-1)\,\sigma,\qquad
\sigma:=\sum_{m=2}^N(-1)^{m}\binom{N-2}{m-2}\,m^{-(j+1)},
\]
which is \eqref{eq:lambda} (note $(-1)^{m-1}=-(-1)^m$). It remains to show $\sigma>0$, which
then gives $\lambda_{N,j}<0$. Reindexing by $k=m-2$,
$\sigma=\sum_{k=0}^{N-2}(-1)^{k}\binom{N-2}{k}(k+2)^{-(j+1)}$, and using
$(k+2)^{-(j+1)}=\tfrac1{j!}\int_0^1 x^{\,k+1}(-\log x)^{\,j}\,dx$ together with
$\sum_{k}(-1)^k\binom{N-2}{k}x^{k}=(1-x)^{N-2}$,
\begin{equation}\label{eq:Tpos}
\sigma=\frac{1}{j!}\int_0^1 x\,(-\log x)^{\,j}\,(1-x)^{N-2}\,dx>0 ,
\end{equation}
since the integrand is strictly positive on $(0,1)$.
\end{proof}

The same eigenvalue is recovered from the radial identity \eqref{eq:final}, confirming
that the two computations agree. Along any tangent direction $\hvec\in\TT$ one has
$\psi'(\theta)=\lambda_{N,j}\|\hvec\|^2\,\theta+o(\theta)$ as $\theta\to0$, since $\uvec$ is
critical; letting $\theta\to0$ in \eqref{eq:final} (where $q_a\to1/N$, $w_{t,a}\to r(t/N)$)
and substituting $s=t/N$ evaluates the limit to
\[
\lambda_{N,j}=-\,\frac{N^3(j-1)}{(j-1)!}\int_0^\infty s^{\,j}e^{-2s}\bigl(1-e^{-s}\bigr)^{N-2}\,ds ,
\]
and expanding $(1-e^{-s})^{N-2}$ by the binomial theorem and integrating term by term returns
\eqref{eq:lambda} and \eqref{eq:Tpos}.

\section{Remarks}\label{sec:remarks}

\begin{remark}[failure of Schur-concavity]\label{rem:schur}
For $N=2$ the function $\E[U_j^N]$ \emph{is} Schur-concave, for every $j\ge2$: on
$\pvec=(p,1-p)$ the closed form \eqref{eq:closed} reduces to
$\E[U_j^2]=2-\bigl(p^{\,j}+(1-p)^{\,j}\bigr)$, with
$\tfrac{d^2}{dp^2}\E[U_j^2]=-j(j-1)\bigl(p^{\,j-2}+(1-p)^{\,j-2}\bigr)<0$, so it is symmetric
and strictly concave in the single free coordinate $p$, hence Schur-concave. (For $j=2$ this
is $1+2p(1-p)$.) Already at $N=3$ and $j=2$, however, it is not. By exact evaluation of \eqref{eq:closed},
\[
\E[U_2^3](0.7,0.2,0.1)=\tfrac{100367}{64800},\qquad
\E[U_2^3](0.6,0.3,0.1)=\tfrac{134447}{88200},
\]
with difference $\tfrac{77891}{3175200}>0$; thus $\E[U_2^3](0.7,0.2,0.1)>\E[U_2^3](0.6,0.3,0.1)$,
even though $(0.6,0.3,0.1)$ is obtained from $(0.7,0.2,0.1)$ by an equalizing (Robin--Hood)
transfer and hence is majorized by it. A Schur-concave function could not decrease under
such a transfer; this already proves the failure. The Schur--Ostrowski criterion
\cite[Sect.~3.A]{MOA} gives the same conclusion through the derivative: a symmetric $C^1$
function is Schur-concave iff $(p_a-p_b)(\partial_{p_a}-\partial_{p_b})\E[U_j^N]\le0$
throughout $\PN$, and this quantity changes sign---it is about $+0.14$ at
$\pvec=(0.7,0.2,0.1)$, $j=2$ for the pair $(1,2)$, yet negative near $\uvec$, where, since
$\uvec$ is critical, along $\pvec=\uvec+\varepsilon(\mathbf{e}_a-\mathbf{e}_b)$ it equals
$4\lambda_{N,j}\varepsilon^2+o(\varepsilon^2)<0$ by Proposition~\ref{cor:hess}. The failure
is not special to $(3,2)$: at $\pvec=\bigl(0.7,\,0.2,\,\tfrac{0.1}{N-2},\dots,\tfrac{0.1}{N-2}\bigr)$,
the Schur--Ostrowski quantity for the pair $(1,2)$, computed exactly from \eqref{eq:closed},
is strictly positive for every $(N,j)$ with $3\le N\le6$ and $2\le j\le4$ (for instance
$+0.136$ at $(3,2)$ and $+0.368$ at $(4,2)$), so $\E[U_j^N]$ is Schur-concave in none of these
cases. The radial derivative \eqref{eq:final}, by contrast, is a single covariance and has a
definite sign.
\end{remark}

\begin{remark}[relation to the companion variance problem and to Long's proofs]\label{rem:long}
The companion extremal problem for the variance of the time to collect $m$ complete sets
(the references are in Section~\ref{sec:intro}) is solved by induction on $N$ for $m=1$ by
Yu \cite{Yu} and by a Poissonized terminal-defect argument for general $m$ by Long
\cite{LongVar}. The present problem differs in kind: $U_j^N$ is a count evaluated at the main
collector's stopping time rather than a completion time, the functional is an expectation,
and the extremum is a maximum. Our route takes the Poissonization from \eqref{eq:DP} and
obtains the sign from Chebyshev's sum inequality after the integration by parts of
Lemma~\ref{lem:radial}. The proofs of Conjecture~\ref{conj:main} by Long
\cite{LongLimit,LongRadial} also differentiate along the radial segment but extract the sign
differently, by rewriting the derivative as a sum of pairwise squared-difference kernels with
a positivity lemma, without an integration by parts or a covariance; Long additionally obtains
monotonicity in $N$, limit laws, and transform-level extremality not treated here.
\end{remark}

\begin{remark}[scope]\label{rem:scope}
Beyond the derivation of \eqref{eq:DP}, which uses $j\ge2$ (Remark~\ref{rem:j1}), the
proof uses $j\ge2$ only through the monotonicity of $x\mapsto x^{\,j-1}$, and $\pvec\in\PN$
only to keep the whole segment in the open simplex, so that $\delta=\min_a p_a>0$ in the
domination Lemma~\ref{lem:dom} and the boundary terms in Lemma~\ref{lem:radial} vanish. Thus, taking \eqref{eq:integral} with $(j-1)!$ read as
$\Gamma(j)$ as the definition of the functional, the argument applies verbatim to every
real $j>1$, since $x\mapsto x^{\,j-1}$ is then strictly increasing on $(0,\infty)$: the
covariance in \eqref{eq:final} is again strictly positive off $\uvec$, so $\uvec$ is the
strict and unique maximizer there as well. Strict radial monotonicity also shows $\uvec$
is the only critical point of the restriction of $\E[U_j^N]$ to $\PN$.
\end{remark}

\subsection*{Use of generative AI}
The authors did not use a generative AI assistant while preparing this article. All definitions, statements and proofs were checked and are the
responsibility of the authors.

\medskip
\noindent\textbf{Acknowledgements.}
We thank the two anonymous referees for their careful reading and constructive
comments, which  led to a number of improvements in the present paper.

\end{document}